\documentclass[10pt,a4paper]{article}

\usepackage{cite}
\usepackage{graphicx}
\usepackage{enumerate}
\usepackage{booktabs}
\usepackage{enumerate}
\usepackage{color}
\usepackage{multirow, array} 
\usepackage{float}
\usepackage{amsmath}
\usepackage{amssymb}
\usepackage{amsthm}
\usepackage{amsfonts}
\usepackage{hyperref}

\newcommand{\Cm}{{\ensuremath{\mathbb{C}^{m\times n}}}}
\newcommand{\Cn}{{\ensuremath{\mathbb{C}^{n\times m}}}}
\newcommand{\Cnn}{{\ensuremath{\mathbb{C}^{n\times n}}}}

\newcommand{\Cmm}{{\ensuremath{\mathbb{C}^{m\times m}}}}

\newcommand{\Ra}{{\ensuremath{\cal R}}}
\newcommand{\Nu}{{\ensuremath{\cal N}}}
\newcommand{\rk}{{\ensuremath{\rm rank}}}

\newcommand{\core}{\mathrel{\text{\textcircled{$\#$}}}}
\newcommand{\odagger}{\mathrel{\text{\textcircled{$\dagger$}}}}

\newcommand{\ind}{{\ensuremath{\text{\rm Ind}}}}

\newcommand{\weak}{\mathrel{\textsc{\textcircled{w}}}}

\newtheorem{theorem}{Theorem}[section]
\newtheorem{corollary}[theorem]{Corollary}
\newtheorem{lemma}[theorem]{Lemma}

\newtheorem{remark}[theorem]{Remark}
\newtheorem{example}[theorem]{Example}
\newtheorem{definition}[theorem]{Definition}

\begin{document}
\author{D.E. Ferreyra\thanks{Universidad Nacional de R\'io Cuarto, CONICET, FCEFQyN, RN 36 Km 601, R\'io Cuarto, 5800, C\'ordoba. Argentina. E-mail: \texttt{deferreyra@exa.unrc.edu.ar, flevis@exa.unrc.edu.ar, pmoas@exa.unrc.edu.ar}.}\;, F.E. Levis$^{*}$, R.P. Moas$^{*}$, H.H. Zhu\thanks{School of Mathematics, Hefei University of Technology, Hefei 230009, China. E-mail: \texttt{hhzhu@hfut.edu.cn}}
}


\title{A new generalized inverse for rectangular matrices. A general approach}
\date{}
\maketitle

\begin{abstract}
Rao and Mitra  in 1972 introduced two different types of constraints to extend the concept of Bott-Duffin inverse  and defined a new constrained inverse. Mary in 2011 defined the inverse along an element that generalizes the Moore-Penrose and Drazin inverses in a semigroup. Drazin in 2012 introduced the $(b,c)$-inverse generalizing the Mary inverse. In 2017, Raki\'c  noted that the Rao-Mitra inverse is a direct precursor of the $(b,c)$-inverse. In this paper, we introduce the notion of $EF$-inverse  as a unified approach to the aforementioned generalized inverses. Moreover, we  show that the recently introduced generalized bilateral inverses that in turn contain the OMP, MPO, and MPOMP inverses can also be considered as special cases of the $EF$-inverse.   
\end{abstract}



\maketitle

\section{Introduction}\label{sec1}

Let ${\mathbb{C}}^{m\times n}$ be the set of $m\times n$ complex matrices and $A\in {{\mathbb{C}}^{m\times n}}$. By $A^{\ast }$, $\rk(A)$, ${\mathcal{N}}(A)$, and ${\mathcal{R}}(A)$ we denote the conjugate transpose, the rank, the null space, and the column space of $A$, respectively. Here,  $P_{\mathcal{M},\mathcal{N}}$  stands for the projector (idempotent) on the subspace $\mathcal{M}$ along the subspace $\mathcal{N}$.
When $\mathcal{N}$ is the orthogonal complement of $\mathcal{M}$, that is,  $\mathcal{N}=\mathcal{M}^{\perp }$, we write $P_{\mathcal{M}}$.

We  recall that the  Moore-Penrose inverse of $A$ is the unique matrix
$X\in {{\mathbb{C}}^{n\times m}}$ such that
\begin{equation*}
AXA=A, \quad XAX=X,\quad (AX)^*=AX,\quad \text{and}\quad
(XA)^*=XA,
\end{equation*} and is denoted by $A^{\dagger }$. The Moore-Penrose inverse can be used to represent the orthogonal projectors $P_{A}:=AA^{\dag }$ and $Q_A:=A^{\dag}A$ onto ${\mathcal{R}}(A)$ and  ${\mathcal{R}}(A^*)$ respectively.

A matrix $X$ verifying the condition $AXA=A$ (or $XAX=X$) is called an inner (or outer) inverse of $A$  and it is denoted by $A^{(1)}$ (or $A^{(2)}$).  The symbol $A\{1\}$ (or $A\{2\}$)  denotes the set of all inner (outer) inverses of $A$. 

The index of $A\in \Cnn$, denoted by $\ind(A)$, is the smallest nonnegative integer $k$ such that $\rk(A^k) = \rk(A^{k+1})$. 

Throughout this paper, we will assume that $\ind(A)=k\ge 1$.

The Drazin inverse of a matrix $A\in {{\mathbb{C}}^{n\times n}}$ of index $k$, is the unique matrix $X\in {{\mathbb{C}}^{n\times n}}$ satisfying  \begin{equation*}
XAX=X,\quad AX=XA,\quad \text{and}\quad A^{k+1}X=A^{k},
\end{equation*}
and is denoted by $A^{d}$. When $k=1$, the Drazin inverse of $A$ is called the group inverse of $A$ and is denoted by $A^{\#}$.

For $A\in \Cm$ of rank $r$, a subspace $\mathcal{T}$ of $\mathbb{C}^n$ of dimension $s\le r$, and a subspace $\mathcal{S}$ of $\mathbb{C}^m$ of dimension $m-s$, the unique matrix $X\in \Cn$ that satisfies 
\[X=XAX, \quad \Ra(X)=\mathcal{T}, \quad \text{and} \quad \Nu(X)=\mathcal{S},\]
is called an outer inverse of $A$ with prescribed range $\mathcal{T}$ and null space $\mathcal{S}$ and is denoted by $A^{(2)}_{\mathcal{T},\mathcal{S}}$. It is well known that $A^{(2)}_{\mathcal{T},\mathcal{S}}$  exists if and only if $A\mathcal{T}\oplus \mathcal{S}=\mathbb{C}^m$.

Recall that the Moore-Penrose inverse, the Drazin inverse, and the group inverse are outer inverses of $A$ with prescribed range and null space satisfying, respectively, the following representations
\begin{equation*}\label{2 inversas tradicionales}
A^\dagger=A^{(2)}_{\Ra(A^*),\,\Nu(A^*)}, \quad A^d=A^{(2)}_{\Ra(A^k),\,\Nu(A^k)}, \quad \text{and} \quad  A^\#=A^{(2)}_{\Ra(A),\,\Nu(A)}.
\end{equation*}

Let $A,B,C\in {{\mathbb{C}}^{n\times n}}$. In 1972, Rao and Mitra \cite{RaMi} introduced two new types of constraints in order to extend the Bott-Duffin inverse \cite{BeGr}:

\begin{center}
\begin{tabular}{ |p{5.4 cm}|p{5.4 cm}|}
\hline
\multicolumn{2}{|c|}{\bf Rao-Mitra constraints} \\
\hline
\textbf{Type I} & \textbf{Type II} \\
\hline
$c:~ \Ra(X)\subseteq \Ra(B)$ & $C:~ XA$ is an identity on $\Ra(B)$ \\
$r:~ \Ra(X^*)\subseteq \Ra(C^*)$ & $R:~ (AX)^*$  is an identity on $\Ra(C^*)$ \\
\hline
\end{tabular}
\end{center} 

\begin{definition}
\label{crCR} \cite{RaMi}  Let $A,B,C\in {{\mathbb{C}}^{n\times n}}$. The $crCR$-inverse of $A$ is a matrix $X\in {{\mathbb{C}}^{n\times n}}$ that satisfies constraints $c$, $r$, $C$ and $R$, and is denoted by $A_{B,C}$.
\end{definition}

In \cite{RaMi} it was proved that $A_{B,C}$ exists if and only if $\rk(CAB)=\rk(C)=\rk(B)$, in which case $A_{B,C}$ is unique. Moreover, the $crCR$-inverse can be computed by using inner inverses: 
 \begin{equation*} \label{obs1} A_{B,C}=B(CAB)^{(1)}C, \quad \text{where} \quad  (CAB)^{(1)}\in CAB\{1\}.
 \end{equation*}

In 2011, Mary \cite{Mary} defined the inverse along an element that contains the classical Moore-Penrose and Drazin inverses (and therefore the group inverse) as particular cases. In 2012, Drazin \cite{Dr2} introduced the $(b,c)$-inverse, which is  more general than the inverse along an element. Both types of generalized inverses were defined in the context of semigroups and abstract rings. Their definitions in the matrix context can be rewritten as follows.

\begin{definition}  \cite{Ra}
\label{allde} Let $A,D \in {{\mathbb{C}}^{n\times n}}$. The \emph{inverse of $A$ along $D$}  is a matrix $X\in {{\mathbb{C}}^{n\times n}}$ such that
$$XAD=D=DAX, \quad {\mathcal{R}}(X)\subseteq {\mathcal{R}}(D), \quad \text{and} \quad {\mathcal{N}}(D)\subseteq {\mathcal{N}}(X).$$  If such a matrix $X$ exists and is unique, it is denoted by $A^{\Vert D}$.
\end{definition}

It follows that the classical generalized inverses are a special case of the Mary inverse: 
\begin{equation*}
A^{\dag }=A^{\Vert A^{\ast }},\quad A^{d}=A^{\Vert A^{k}},\quad \text{and}\quad A^{\#}=A^{\Vert A}.
\end{equation*}

\begin{definition}\cite{Dr2,Ra}
\label{allde2} Let $A,B,C\in {{\mathbb{C}}^{n\times n}}$. The $(B,C)$-inverse of $A$  is a matrix $X\in {{\mathbb{C}}^{n\times n}}$ such that
\begin{equation} \label{BC inverse}
XAB=B,\quad CAX=C,\quad {\mathcal{R}}(X)\subseteq {\mathcal{R}}(B),\quad \text{and}\quad {\mathcal{R}}(X^{\ast })\subseteq {\mathcal{R}}(C^{\ast }).
\end{equation}
If such a matrix $X$ exists and is unique, it is denoted by $A^{\Vert (B,C)}$.
\end{definition}

As proved in \cite{Dr2}, the inverse along an element is a particular case of the $(B,C)$-inverse, that is, $$A^{\Vert (D,D)}=A^{\Vert D}.$$

Raki\'{c} \cite{Ra} noted that the $crCR$-inverse is a precursor of the $(b,c)$-inverse in  a ring. Moreover, in the matrix case both generalized inverses coincides, that is,  $$A^{\Vert (B,C)}=A_{B,C}.$$

In this paper we introduce a new type of generalized
inverse for a rectangular matrix  to be called $EF$-inverse
which generalizes the $(B,C)$-inverse and therefore, the classical Moore-Penrose, Drazin, and group inverses as well as all the generalized inverses recently appeared in the literature. This paper is organized as follows. In Section 2 we introduce the $EF$-inverse and we analyse its existence and uniqueness as solution of a matrix system.
Also, we give some of its main properties.  Section 3 is devoted to the study the canonical form of the $EF$-inverse. In section 4 we show that the more recent generalized inverses such as the generalized bilateral inverses which in turn include the OMP, MPO, and MPOMP inverses are particular cases of the $EF$-inverse.

\section{The $EF$-inverse. An extension of the Rao-Mitra inverse}

It is well known that the Moore-Penrose inverse $A^\dag$ of $A\in \Cm$ is the unique matrix $X\in \Cn$ satisfying 
\[XAX=X, \quad XA=P_{\Ra(A^*)}, \quad \text{and} \quad AX=P_{\Ra(A)}.\]
Similarly, the Drazin inverse $A^d$ of $A\in \Cnn$ is the unique $X\in \Cnn$ such that 
\[XAX=X \quad \text{and} \quad AX=XA=P_{\Ra(A^k),\Nu(A^k)}.\]
In particular, the group inverse of a matrix $A$ of index 1 is the unique matrix $X$ that satisfies $XAX$, $AX=XA=P_{\Ra(A),\Nu(A)}$.
 
Also, recall that the  generalized inverse with prescribed range and null space $A^{(2)}_{\mathcal{T},\mathcal{S}}$ of $A$ (if exists) is the unique matrix $X$ such that 
\begin{equation}\label{outer 2TS}
XAX=X, \quad XA=P_{\mathcal{T},(A^*(\mathcal{S}^{\perp}))^{\perp}}, \quad \text{and} \quad AX=P_{A(\mathcal{T}),\mathcal{S}}.
\end{equation}

In 2010, the core inverse for a square matrix was introduced  by Baksalary and Trenkler \cite{BaTr}.  
For a given matrix $A \in \Cnn$, the core inverse of $A$ is defined to be a matrix $X\in \Cnn$ satisfying the conditions 
\[AX=P_A \quad \text{and}\quad \Ra(X)\subseteq \Ra(A).\]
The authors  proved that $A$ is core invertible if and only if $\ind(A)= 1$. In this case, the core inverse (or GMP) of $A$ is the unique matrix  given by $A^{\core}=A^\# A A^\dag$. 
  
It is easy to see that the core inverse can be characterized by the following conditions
\[XAX=X, \quad XA=P_{\Ra(A),\Nu(A)}  \quad \text{and} \quad AX=P_{\Ra(A),\Nu(A^*)}. \]

In each of matrix system previous, we note that they possess a similar characterizing property in terms of suitable projectors (idempotent/orthogonal). That is, $X$ is a generalized inverse of  $A$ if satisfies the following equations 
\begin{equation}\label{system EF}
XAX=X, \quad  XA=E, \quad \text{and} \quad AX=F, 
\end{equation}
for certain projectors $E$ and $F$.  

Motivated by \eqref{system EF}  we introduce a new type of generalized inverse for rectangular matrices that extends the notion of $crCR$-inverse.

We begin with two lemmas due to Penrose.
\begin{lemma}
\cite{Pe, RaMi}  \label{lemma 1}  Let  $A \in \mathbb{C}^{m \times n}$, $B \in \mathbb{C}^{p \times q}$, $C \in \mathbb{C}^{m \times q}$, $A^{(1)}\in A\{1\}$, and $B^{(1)}\in B\{1\}$.   Then, the equation $AXB=C$ is consistent (in $X$) if and only if $AA^{(1)}CB^{(1)}B=C$, in which case the general solution is
\[X=A^{(1)}CB^{(1)}+Z-A^{(1)}AZBB^{(1)},\] where $Z$ is an arbitrary matrix.
\end{lemma}

\begin{lemma}
\cite{BeGr} \label{lemma 2} Let $A \in \mathbb{C}^{m \times n}$, $B \in \mathbb{C}^{p \times q}$, $D \in \mathbb{C}^{m \times p}$, and $E  \in \mathbb{C}^{n \times q}$. The matrix equations 
\begin{equation} \label{mqs}
AX=D \quad  \text{and}\quad XB=E,
\end{equation}
have a common solution if and only if each equation separately has a solution and  $AE=DB$. In particular, if $X_{0}  \in \mathbb{C}^{n \times p}$ is a solution of \eqref{mqs},  the general solution is $$X=X_{0}+(I_n-A^{(1)}A)Y(I_p-BB^{(1)}),$$ for arbitrary $A^{(1)}\in A\{1\}$, $B^{(1)}\in B\{1\}$, and $Y \in \mathbb{C}^{n \times p}$.
\end{lemma}

\begin{definition}
\label{EFinv}  Let $A\in \Cm$, $E\in \Cnn$, and $F\in \Cmm$.  An $EF$-inverse of $A$  is a matrix $X\in \Cn$ that verifies  \begin{equation}\label{system 1}
 XAX=X, \quad XA=E,\quad \text{and} \quad  AX=F.
\end{equation}
If such a matrix $X$ exists, it is denoted by $A^{(E,F)}$.
\end{definition}

We next establish the existence and the uniqueness of the introduced EF-inverse.

\begin{theorem} \label{existence uniqueness} Let $A\in \Cm$, $E\in \Cnn$, and $F\in \Cmm$.  The matrix equations
given in \eqref{system 1}
have a common solution if and only if 
\begin{equation}\label{condition 1}
EA^{(1)}A=E, \quad AA^{(1)}F=F, \quad  AE=FA, \quad E^2=E, \quad\text{and}\quad F^2=F,
\end{equation} 
 for some $ A^{(1)}\in A\{1\}$, in which case, the unique solution is given by  
\[
A^{(E,F)}=EA^{(1)}F.
\] 
\end{theorem}

\begin{proof} $\Rightarrow)$ Since \eqref{system 1} is consistent, from  Lemma \ref{lemma 2} it follows that  each  of the matrix equations $XA=E$ and $AX=F$   has a solution in $X$ and  verify $AE=FA$. By applying Lemma \ref{lemma 1} to each of the above equations results that $EA^{(1)}A=E$ and $AA^{(1)}F=F$, respectively, for some $ A^{(1)}\in A\{1\}$. Clearly, $E$ and $F$ are idempotent since $XAX=X$. Thus, \eqref{condition 1} is satisfied. \\
$\Leftarrow)$ Consider the five conditions given in \eqref{condition 1}. It suffices to check that the matrix $X:=EA^{(1)}F$  satisfies the matrix equations \eqref{system 1}. In fact, \[XAX=EA^{(1)}FAEA^{(1)}F=(EA^{(1)}A)E^2A^{(1)}F=E^3A^{(1)}F=EA^{(1)}F=X.\] 
The matrix equations $XA=E$ and $AX=F$ can be verified similarly. 
\\ Finally, it remains to prove uniqueness. Let $X_1$ and $X_2$ be two  matrices satisfying \eqref{system 1}. Therefore, 
\begin{equation*}X_1=(X_1A)X_1=EX_1=(X_2A)X_1=X_2(AX_1)=X_2F=X_2(AX_2)=X_2. 
\end{equation*}
The proof is complete. 
\end{proof}

\begin{corollary}\label{corollary} Let $A\in \Cm$, $E\in \Cnn$, and $F\in \Cmm$. If $A^{(E,F)}$ exists, then 
\begin{equation*}
A^{(E,F)}=EX_1=X_2F,
\end{equation*}
where $X_1,X_2\in A\{2\}$.
\end{corollary}
\begin{proof}
By Theorem \ref{existence uniqueness} we have $A^{(E,F)}=EA^{(1)}F$. Now, we consider $X_1:=A^{(1)} F$. Then, by using the second equation in \eqref{condition 1} we obtain \[X_1AX_1=A^{(1)} F A A^{(1)} F=A^{(1)} F^2=A^{(1)} F=X_1.\] 
Thus, $A^{(E,F)}=EX_1$ with $X_1\in A\{2\}$. \\
Similarly, by taking $X_2:=EA^{(1)}$, from the first equation in  \eqref{condition 1} we deduce that $X_2=X_2AX_2$. 
\end{proof}

\begin{remark}\label{rem 1} The condition \eqref{condition 1} in Theorem \ref{existence uniqueness} is equivalent to 
\[EA^{(1)}A=E, \quad AA^{(1)}F=F, \quad  AE=FA, \quad E^2=E, \quad\text{and}\quad F^2=F,\]
for an arbitrary $A^{(1)}\in A\{1\}$. 
\end{remark}

\begin{theorem} \label{characterization EF} Let $A\in \Cm$, $E\in \Cnn$, and $F\in \Cmm$.  The following statements are equivalent:
\begin{enumerate}[(a)]
\item $A^{(E,F)}$ exists;
\item $EA^{(1)}A=E$, $AA^{(1)}F=F$, $AE=FA$, $E^2=E$, and $F^2=F$,  for some $ A^{(1)}\in A\{1\}$;
\item $\Nu(A)\subseteq \Nu(E)$, $\Ra(F)\subseteq \Ra(A)$, $AE=FA$, $E^2=E$, and $F^2=F$;
\item $\Ra(E^*)\subseteq \Ra(A^*)$, $\Ra(F)\subseteq \Ra(A)$,  $AE=FA$, $E^2=E$, and $F^2=F$;
\item $A^{(2)}_{\Ra(E),\Nu(F)}$ exists, $\Nu(FA)=\Nu(E)$, and $\Ra(AE)=\Ra(F)$.
\end{enumerate}
Moreover, in this case 
\begin{equation}\label{representation EF 1}
A^{(E,F)}=EA^\dag F=A^{(2)}_{\Ra(E),\Nu(F)}.
\end{equation}
\end{theorem}

\begin{proof}
(a) $\Leftrightarrow$ (b) Follows from Theorem \ref{existence uniqueness}. \\
(b) $\Leftrightarrow$ (c) Note that the equation $EA^{(1)}A=E$ is equivalent to $\Ra(I_n-A^{(1)}A)\subseteq \Nu(E)$ which in turn is equivalent to $\Nu(A^{(1)}A)=\Nu(A)\subseteq \Nu(E)$. Similarly,   $AA^{(1)}F=F$ is true if and only if $\Ra(F)\subseteq \Ra(A)$.  \\
(c) $\Leftrightarrow$ (d) It is consequence of the classical property $\Nu(A)=\Ra(A^*)^\perp$. \\
(a) $\Rightarrow$ (e)
Let $X:=A^{(E,F)}$. As $X$ is an outer inverse we have 
\[\Nu(X)=\Nu(AX)=\Nu(F) \quad \text{and} \quad \Ra(X)=\Ra(XA)=\Ra(E).\] Therefore, 
\begin{equation}\label{Equality 1}
A^{(E,F)}=A^{(2)}_{\Ra(E),\Nu(F)}. 
\end{equation}
From equivalence (a) $\Leftrightarrow$ (d)   we get 
\[\Ra(F)=\Ra(F^2)=F\Ra(F)\subseteq F\Ra(A)=\Ra(FA)=\Ra(AE)\subseteq \Ra(F).\] 
Similarly, as (a) $\Leftrightarrow$ (c) we obtain 
\[\Ra(E^*)=\Ra((E^*)^2)=E^*\Ra(E^*)\subseteq E^*\Ra(A^*)=\Ra((AE)^*)=\Ra((FA)^*),\] which implies $\Nu(FA)\subseteq \Nu(E)\subseteq \Nu(AE)=\Nu(FA)$. \\
(e) $\Rightarrow$ (a)  Let $X:=A^{(2)}_{\Ra(E),\Nu(F)}$ be such that $\Nu(FA)=\Nu(E)$ and $\Ra(AE)=\Ra(F)$.  From \eqref{outer 2TS}, it follows that $XAX=X$, 
\begin{eqnarray*}
XA &=& P_{\Ra(E),(A^*(\Nu(F)^{\perp}))^{\perp}} \\
&=& P_{\Ra(E),\Ra((FA)^*)^{\perp}} \\
&=& P_{\Ra(E),\Nu(FA)} \\ 
&=& P_{\Ra(E),\Nu(E)}=E,
\end{eqnarray*}
and 
\begin{eqnarray*}
AX &=& P_{A(\Ra(E)),\Nu(F)} \\
&=& P_{\Ra(AE),\Nu(F)} \\
&=& P_{\Ra(F),\Nu(F)}=F.
\end{eqnarray*}
Finally,  the first equality in \eqref{representation EF 1} follows from Remark \ref{rem 1}, while the second equality is due to \eqref{Equality 1}. 
\end{proof}

An example of how to calculate the $EF$-inverse is shown below.

\begin{example}\label{example 1} Consider the matrix
\[A= \left[\begin{array}{rrr}
a & 0 & 0 \\
0 & a & 0  \\
0 & 0 & 0   \\
0 & 0 & 0 
\end{array}\right], \quad a\in \mathbb{R} \setminus\{0\},\]
in conjunction with the projectors 
\[E= \left[\begin{array}{rrr}
\frac{1}{2} & \frac{1}{2} & 0 \\
\frac{1}{2} &  \frac{1}{2} & 0  \\
0 &  0 &  0  
\end{array}\right] \quad \text{and}\quad F=\left[\begin{array}{rrrr}
\frac{1}{2} & \frac{1}{2} & 0 & 0\\
\frac{1}{2} &  \frac{1}{2} & 0  & 0 \\
0 &  0 &  0  &0 \\
0 & 0 & 0 & 0
\end{array}\right].\]
These matrices satisfy $EA^\dag A=E$, $AA^\dag F=F$, $AE=FA$, $E^2=E$, and $F^2=F$, which is a guarantee for the existence of the $EF$-inverse of $A$ given by
\[
A^{(E,F)}=EA^\dag F=\left[\begin{array}{rrrr}
\frac{1}{2a} & \frac{1}{2a} & 0 & 0\\
\frac{1}{2a} & \frac{1}{2a} & 0 & 0 \\
0 & 0 & 0  & 0
\end{array}\right].
\]
\end{example}

Now, we show that the $crCR$-inverse is a particular case of the $EF$-inverse. Before, we prove a new characterization of the  $crCR$-inverse.  

\begin{theorem}\label{system BC inverse} Let $A,B,C\in \Cnn$. If $A_{B,C}$ exists, then it is the unique matrix $X\in \Cnn$ such that 
\begin{equation}\label{XA and AX}
XAX = X,\quad XA=P_{\Ra(B),\Nu(CA)},\quad \text{and} \quad AX=P_{\Ra(AB),\Nu(C)}.
\end{equation}
\end{theorem}

\begin{proof}
Let $X:=A_{B,C}$ satisfying \eqref{BC inverse}. It follows that $\Ra(X)=\Ra(B)$ and $\Nu(X)=\Nu(C)$.  Since $\Ra(X)\subseteq \Ra(B)$, there exists some matrix $D$ such that $$X=BD=(XAB)D=XA(BD)=XAX.$$ Consequently,
\begin{equation} \label{f1} \Ra(XA)=\Ra(X)=\Ra(B)\quad \text{and} \quad \Nu(AX)=\Nu(X)=\Nu(C).
\end{equation}
Now, according to \eqref{f1} we have $$\Ra(AX)= \Ra(AB) \quad \text{and} \quad  \Nu(XA)=\Nu(CA).$$
Since $XAX=X$, clearly $XA$ and $AX$ are idempotent. Therefore, from the uniqueness of an oblique projector we obtain the last two equations in \eqref{XA and AX}. \\
Finally, the uniqueness follows similarly to the uniqueness proof of Theorem  \ref{existence uniqueness}. 
\end{proof}

\begin{corollary}
\label{cor 1} Let $A,B,C\in {{\mathbb{C}}^{n\times n}}$. If $A_{B,C}$ exists, then  $A_{B,C}=A^{(E,F)}$,
where  $E=P_{{\mathcal{R}}(B),{\mathcal{N}}(CA)}$ and $F=P_{{
\mathcal{R}}(AB),{\mathcal{N}}(C)}$.
\end{corollary}
\begin{proof} Assume that $A_{B,C}$ exists. By Theorem \ref{system BC inverse} we have  that $X:= A_{B,C}$ verifies $XAX = X$, $XA=E$,  and $AX=F$. The rest of the equalities are a consequence of Theorem \ref{existence uniqueness} and Definition \ref{EFinv}. 
\end{proof}

\begin{corollary}
Let $A,E,F \in \Cnn$. If $A^{(E,F)}$ exists, then $A_{E,F}$ exists and $A^{(E,F)}=A_{E,F}$.
\end{corollary}
\begin{proof} If $A^{(E,F)}$ exists, by Theorem \ref{characterization EF}, $A^{(2)}_{R(E),N(F)}$ exists  and \linebreak $A^{(E,F)}=A^{(2)}_{R(E),N(F)}$. Then, from Theorem 1.5 in \cite{Ra}, $A_{E,F}$ exists and \linebreak  $A^{(2)}_{R(E),N(F)}=A_{E,F}$. Thus $A^{(E,F)}=A_{E,F}$. 
\end{proof}

The following example illustrates the previous corollary. 
\begin{example}\label{example 2} Consider the matrix
\[A= \left[\begin{array}{rrr}
a & 0 & 0 \\
0 & b & 0  \\
0 & 0 & 0    
\end{array}\right], \quad a,b\in \mathbb{R} \setminus\{0\},\]
in conjunction with the projectors 
\[E= \left[\begin{array}{rrr}
1 & 1 & 0 \\
0 &  0 & 0  \\
0 &  0 &  0  
\end{array}\right] \quad \text{and}\quad F=\left[\begin{array}{rrr}
1 & \frac{a}{b} & 0 \\
0 &  0 & 0   \\
0 &  0 &  0   
\end{array}\right].\]
These matrices satisfy $EA^\dag A=E$, $AA^\dag F=F$, $AE=FA$, $E^2=E$, and $F^2=F$, which is a guarantee for the existence of the $EF$-inverse of $A$ given by
\[
A^{(E,F)}=EA^\dag F=\left[\begin{array}{rrr}
\frac{1}{a} & \frac{1}{b} & 0 \\
0 & 0 & 0  \\
0 & 0 & 0  
\end{array}\right].
\] 
On the other hand, the $crCR$-inverse is given by
\[
A_{E,F}=E(FAE)^\dag F=\left[\begin{array}{rrr}
\frac{1}{a} & \frac{1}{b} & 0 \\
0 & 0 & 0 \\
0 & 0 & 0  
\end{array}\right].
\]
\end{example}

We finish this section with a result about when the $EF$-inverse is  an inner inverse  of the matrix. 

\begin{theorem}\label{1} Let $A\in \Cm$, $E\in \Cnn$, and $F\in \Cmm$. If $A^{(E,F)}$ exists, then the following statements are equivalent:
\begin{enumerate}[(a)]
\item  $A^{(E,F)}\in A\lbrace 1 \rbrace$;
\item  $A=AE$;
\item  $A=FA$;
\item  $\Ra(A)\subseteq \Ra(F)$;
\item  $\Nu(E)\subseteq \Nu(A)$.
\end{enumerate}
\end{theorem}

\begin{proof}
(a) $\Rightarrow$ (b). By \eqref{system 1} we get $A=AA^{(E,F)}A=AE$. \\
(b) $\Rightarrow$ (a). From  Theorem \ref{characterization EF} we have $AA^{(E,F)}A=AEA^\dag FA=AA^\dag A=A$. 
\\ (b) $\Leftrightarrow$ (c). Follows from Theorem \ref{characterization EF}. 
\\ (c) $\Leftrightarrow$ (d). Note that $A=FA$ holds if and only if  $(I_m-F)A=0$, which in turn is equivalent to  $\Ra(A) \subseteq \Nu(I_m-F)=\Ra(F)$. \\
(d) $\Leftrightarrow$ (b). Clearly, $A=AE$ holds if and only if $\Nu(E)=\Ra(I_n-E)\subseteq \Nu(A)$.  \\ The proof is complete. 
\end{proof}

\section{Canonical form of the $EF$-inverse}

In this section we exhibit an interesting result about canonical form of the $EF$-inverse of a rectangular matrix. 
  The tool which we use is the Singular Value Decomposition (SVD).  We then discuss immediate consequences of a such canonical form. We also give some more characterizations of the $EF$-inverse.

\begin{theorem}(SVD)\label{svd} Let $A\in \Cm$ be a nonnull matrix of rank $r>0$ and let $\sigma_1\ge\sigma_2\ge\cdots \ge \sigma_r>0$ be the singular values of $A$.  Then there exist unitary matrices $U\in \Cmm$ and $V\in \Cnn$ such that
$A=U\begin{bmatrix}                                                                              \Sigma & 0 \\
0 & 0
\end{bmatrix}V^*$, where $\Sigma=diag(\sigma_1, \sigma_2, \hdots, \sigma_r)$. In particular, the Moore-Penrose inverse of $A$ is given by
\begin{equation*}\label{MP respecto to SVD}
A^\dag=V\begin{bmatrix}                                                                              \Sigma^{-1} & 0 \\
0 & 0
\end{bmatrix}U^*.
\end{equation*}
\end{theorem}

\begin{theorem}\label{canonical form rect} Let $A\in \Cm$ be written as in Theorem  \ref{svd} and let 
\begin{equation}\label{projectors E and F}
E=V\begin{bmatrix}
         E_1 & E_2 \\
          E_3 & E_4 \\
        \end{bmatrix}V^*\in \Cnn, \quad F=U\begin{bmatrix}
         F_1 & F_2 \\
          F_3 & F_4 \\
        \end{bmatrix}U^*\in \Cmm.
\end{equation}
 Then the following statements are equivalent:
\begin{enumerate}[(a)]
\item  $A^{(E,F)}$ exists;
\item  $E_1^2=E_1$, $F_1^2=F_1$, $\Sigma E_1=F_1\Sigma$, $E_3E_1=E_3$, $F_1F_2=F_2$, $E_2=0$, $E_4=0$, $F_3=0$, and $F_4=0$;
\item  $E_1^2=E_1$, $F_1^2=F_1$, $\Sigma E_1=F_1\Sigma$, $\Nu(E_1)\subseteq \Nu(E_3)$, $\Ra(F_2)\subseteq \Ra(F_1)$, $E_2=0$, $E_4=0$, $F_3=0$, and $F_4=0$.
\end{enumerate}
In this case, 
\begin{equation}\label{canonical form}
 A^{(E,F)}= V\begin{bmatrix}
               E_1 \Sigma^{-1} & E_1 \Sigma^{-1} F_2\\
               E_3 \Sigma^{-1} & E_3 \Sigma^{-1} F_2 \\
              \end{bmatrix} U^*.
\end{equation}
\end{theorem}
\begin{proof}
(a) $\Rightarrow $ (b). Since $A^{(E,F)}$ exists,
from Definition \ref{EFinv} and Remark \ref{rem 1} we have 
\begin{equation}\label{condition 2}
EQ_A=E, \quad P_AF=F, \quad  AE=FA, \quad E^2=E, \quad\text{and}\quad F^2=F.
\end{equation} 
The first equality in \eqref{condition 2} is equivalent to $E_2=0$ and $E_4=0$. While the second equality is equivalent to $F_3=0$ and $F_4=0$. Thus, 
\begin{equation}\label{E and F}
E=V\begin{bmatrix}
         E_1 & 0 \\
          E_3 & 0 \\
        \end{bmatrix}V^*\in \Cnn, \quad F=U\begin{bmatrix}
         F_1 & F_2 \\
          0 & 0 \\
        \end{bmatrix}U^*\in \Cmm.
\end{equation}
In consequence, from $E^2=E$  we obtain that $E_1^2=E_1$  and $E_3E_1=E_3$. Similarly,  from $F^2=F$ we have  $F_1^2=F_1$ and $F_1F_2=F_2$. Now, by using the expressions for $E$ and $F$ given in \eqref{E and F} and the condition $AE=FA$ we obtain $\Sigma E_1=F_1\Sigma$.
\\ (b) $\Rightarrow $ (a). It easy to check that the matrix given in \eqref{canonical form} satisfies the three conditions in \eqref{system 1}.
\\ (b) $\Leftrightarrow$ (c). Note that under assumption $E_1^2=E_1$, the equality $E_3E_1=E_3$ is equivalent to $\Nu(E_1)\subseteq \Nu(E_3)$. In fact, $E_3E_1=E_3$  holds if and only if $E_3(I_r-E_1)=0$ which in turn is equivalent to  $\Nu(E_1)\subseteq \Nu(E_3)$. Similarly, $F_1F_2=F_2$ is equivalent to $\Ra(F_2)\subseteq \Ra(F_1)$ provided  $F_1^2=F_1$. 
\\ Finally, in order to prove \eqref{canonical form} we use the representation of the $EF$-inverse of $A$ obtained in \eqref{representation EF 1} and the fact $F_1=\Sigma E_1 \Sigma^{-1}$, $E_3E_1=E_3$, and $E_1^2=E_1$. 
\end{proof}

\begin{example}\label{example 3} Consider the matrix
\[A= \left[\begin{array}{rrr}
a & 0 & 0 \\
0 & b & 0  \\
0 & 0 & 0   \\
0 & 0 & 0 
\end{array}\right], \quad a\in \mathbb{R} \setminus \{0\}, b\in \mathbb{R}.\]
in conjunction with the projectors 
\[E= \left[\begin{array}{rrr}
1 & 0 & 0 \\
c &  0 & 0  \\
0 &  0 &  0  
\end{array}\right] \quad \text{and}\quad F=\left[\begin{array}{rrrr}
1 & 0 & 0 & 0\\
\frac{bc}{a} &  0 & 0  & 0 \\
0 &  0 &  0  &0 \\
0 & 0 & 0 & 0
\end{array}\right].\]
Note that the matrix  $A$ is written as in Theorem \ref{svd} with $U=I_4$, $V=I_3$, and 
$\Sigma=\begin{bmatrix}
a & 0 \\
0 & b
\end{bmatrix}$. So, from \eqref{projectors E and F} one can see that 
$E_1=\begin{bmatrix}
1 & 0 \\
c & 0
\end{bmatrix}$, $E_2=0$, $E_3=0$, $E_4=0$,  $F_1=\begin{bmatrix}
1 & 0 \\
\dfrac{bc}{a} & 0
\end{bmatrix}$, and $F_2=F_3=F_4=0$. 
These matrices satisfy the conditions in (b) of Theorem \ref{canonical form rect}, which is a guarantee for the existence of the $EF$-inverse of $A$. Thus, from \eqref{canonical form} it is easy to check
\[
A^{(E,F)}=\left[\begin{array}{rrrr}
\frac{1}{a} & 0 & 0 & 0\\
\frac{c}{a} & 0 & 0 & 0 \\
0 & 0 & 0  & 0
\end{array}\right].
\]
\end{example}

\section{The $EF$-inverse as an extension of generalized bilateral and non-bilateral  inverses}

The core inverse for a square matrix was introduced in \cite{BaTr} by Baksalary and Trenkler as recently as 2010. Since then a considerable amount of research has been added to advance the theory of this inverse \cite{FeLeTh1, Ku, Ma, WaLi}.

As the core inverse exists only for matrices of index at most 1, in 2014 three kinds of generalizations of the core inverse were defined for complex square matrices of an arbitrary index.
We recall its definitions. Let $A \in \Cnn$ be with $\ind(A)=k$. Then
the unique matrix $X \in \Cnn$ satisfying
\[
XAX=X\quad \text{and} \quad
\Ra(X)=\Ra(X^*)=\Ra(A^k),
\] is called the core-EP inverse (or CEP) of $A$ and is
denoted by $A^{\odagger}$ \cite{MoPr}. In the same way, the dual core-EP (or $*$CEP) was defined as the unique matrix satisfying 
\[XAX=X\quad \text{and} \quad \Ra(X)=\Ra(X^*)=\Ra((A^k)^*),\]
and is denoted by $A_{\odagger}$.

The DMP inverse of $A$ is the unique matrix $X:=A^dAA^\dag$ that satisfies
\[
XAX=X,\qquad
XA=A^d A, \quad \text{and} \quad A^kX= A^k A^{\dagger},
\] is called
and is represented by $A^{d,\dagger}$ \cite{MaTh}. The associated dual inverse is given by the matrix $A^{\dag,d}=A^\dag A A^d$ and is called $*$DMP (or MPD) inverse of $A$.

The unique matrix given by $A^\diamond=(AP_A)^\dag$ is called the BT inverse of $A$ \cite{BaTr2}.

In 2018, the CMP inverse of a square matrix was presented by Mehdipour and Salemi \cite{MeSa}  as the unique matrix $X:=A^\dag AA^dAA^\dag$ (denoted by $A^{c,\dag}$) that satisfies 
\[
XAX=X,\quad AXA=AA^dA, \quad XA=A^\dag AA^dA, \quad AX=AA^dAA^\dag.
\] 

In the same year,  Wang and Chen \cite{WaCh} introduced the  WG inverse of a matrix \linebreak $A \in \Cnn$  as the unique matrix $X:=A^{\weak}$ satisfying
\[
AX^2=X \quad  \text{and} \quad AX=A^{\odagger}A.
\]
If $\ind(A)=1$, the WG inverse and the group inverse coincide. 

Similarly, by using the core-EP and dual core-EP inverses, Chen et al. \cite{ChMoXu} defined the MPCEP and $*$CEPMP inverses of $A$ as the matrices $A^{\dag,\odagger}=A^\dag A A^{\odagger}$ and $A_{\odagger,\dag}=A_{\odagger}AA^\dag$, respectively. 

To extend and unify most of above mentioned definitions of generalized inverses, the OMP, MPO, and MPOMP inverses were defined in \cite{MoSt} composing an arbitrary outer inverse and the Moore-Penrose inverse. More precisely, the OMP inverse of  $A\in \Cm$ is defined as the unique matrix $X:=A^{(2)}_{T,S}AA^\dag \in \Cn$ such that
\[
XAX=X,\quad XA=A^{(2)}_{T,S}A, \quad AX=AA^{(2)}_{T,S}AA^\dag,
\] 
and is denoted by  $A^{(2),\dag}_{T,S}$. Clearly, the core, DMP, and $*$CEPMP inverses are particular cases of the OMP inverse. 
Dually, the MPO (or $*$OMP) inverse of $A$ is the matrix  $A^{\dag,(2)}_{T,S}:=A^\dag A A^{(2)}_{T,S}$, which extends the dual core (or MPG), MPD and MPCEP inverses. 

On the other hand, the MPOMP inverse of $A$ given by the matrix $A^{\dag,(2),\dag}_{T,S}:=A^\dag A A^{(2)}_{T,S}AA^\dag$ generalizes the CMP and Moore-Penrose inverses. Notice that the \linebreak MPOMP inverse of $A$ can be rewritten in terms of the OMP (or MPO) inverse  as  $A^{\dag,(2),\dag}_{T,S}=A^\dag A A^{(2),\dag}=A^{\dag,(2)}AA^\dag$.

Motived by the way in which some of these inverses were defined, recently Kheirandish and Salemi  introduced the notion of generalized bilateral inverse as a unified approach to such inverses. 

The following definition is a slight modification of \cite[Definition 2.1]{KhSa} according to characterizations presented by the authors in Theorems 2.5 and 2.6. 
\begin{definition}  
Let $A\in \Cm$ and let $X_1,X_2 \in \Cn$ be such that $X_1 \in A\{2\}$ and $X_2 \in A\{1\}$. Then $X_1AX_2$ (or $X_2AX_1$) is called generalized bilateral inverse of $A$. 
\end{definition}

\begin{remark} By Theorems 2.5 and 2.6 in \cite{KhSa}, we know that the generalized bilateral inverse of a matrix $A\in \Cm$ always exists and is unique.
\end{remark}

Next, we show that the generalized bilateral inverse of a matrix can be obtained as a particular case of the $EF$-inverse.

\begin{theorem}\label{GBI 1} Let $A\in \Cm$ and let $X_1,X_2 \in \Cn$ be such that $X_1 \in A\{2\}$ and $X_2 \in A\{1\}$. Then  $A^{(E,F)}=X_1AX_2$, where $E=X_1A$ and $F=AX_1AX_2$.
\end{theorem}

\begin{proof}
Let $X:=X_1AX_2$. From \cite[Theorems 2.5]{KhSa} we deduce
that $X$ is the unique matrix such that 
\begin{equation} \label{f1b} 
XAX = X,\quad XA=X_1A=E,\quad \text{and} \quad AX=AXAX=AX_1AX_2=F.
\end{equation}
Therefore,  $A^{(E,F)}$ exists and $X=A^{(E,F)}$. 
\end{proof}

Applying the same method as in Theorem \ref{GBI 1} and by using \cite[Theorems 2.6]{KhSa} instead of \cite[Theorems 2.5]{KhSa}, we obtain the following result.

\begin{theorem}\label{GBI 2} Let $A\in \Cm$ and let $X_1,X_2 \in \Cn$ be such that $X_1 \in A\{2\}$ and $X_2 \in A\{1\}$. Then  $A^{(E,F)}=X_2AX_1$, where $E=X_2AX_1A$ and $F=AX_1$.
\end{theorem}

We would like to point out that the BT, CEP, and WG inverses are not generalized bilateral inverses.  

In the following tables we illustrate that both generalized bilateral and non-\-bilateral inverses are particular cases of the $EF$-inverse.   For the sake of completeness,  we also add definitions and notations of some more  generalized inverses studied recently in the literature. 

\begin{table}[hb!]
\centering
\resizebox{12cm}{!} {
\begin{tabular}{l l l l l}
\hline
\textbf{Generalized bilateral inverses} &  \textbf{Name} & \textbf{$XA=E$} &   \textbf{$AX=F$}  & \textbf{Reference} \\  \hline \addlinespace
$A^{\core}=A^{\#}P_A$ & GMP & $A^{\#}A$ & $P_A$ &  \cite{BaTr}    \\ \addlinespace
$A_{\core}=Q_AA^{\#}$ & MPG & $Q_A$ & $AA^{\#}$ &  \cite{MoPr}   \\
\addlinespace
$A^{d,\dagger}=A^dP_A$ & DMP  & $A^dA$  & $AA^dP_A$ &  \cite{MaTh} \\
\addlinespace
$A^{\dagger, d}=Q_A A^d$ &  MPD  & $Q_AA^dA$   & $AA^d$  &  \cite{MaTh}  \\
\addlinespace
$A^{c,\dag}=A^{\dag,d} P_A$ &  CMP & $Q_AA^dA$ & $   AA^dP_A $  & \cite{MeSa} \\ 
\addlinespace
$A^{\dag, \odagger}=Q_A A^{\odagger}$ &  MPCEP & $Q_AA^{\odagger}A$ & $AA^{\odagger}$ &  \cite{ChMoXu}   \\
\addlinespace
$A_{\odagger,\dag}=A_{\odagger}P_A$ &  $*$CEPMP & $A_{\odagger}A$ & $AA_{\odagger}P_A$ &  \cite{ChMoXu}   \\
\addlinespace
$A^{\weak,\dag}=A^{\weak} P_A$ &  WGMP &  $A^{\weak}A$ &  $A A^{\weak}P_A$ &  \cite{FeLePrTh}   \\ \addlinespace
$A^{\dag,\weak}=Q_A A^{\weak}$ & MPWG &  $Q_A A^{\weak}A$  & $AA^{\weak}$ &  \cite{FeLePrTh} \\  \addlinespace
$A_{T,S}^{(2),\dag}=A_{T,S}^{(2)}P_A$ & OMP &  $A_{T,S}^{(2)}A$ & $AA_{T,S}^{(2)}P_A$ &  \cite{MoSt} \\ \addlinespace
$A_{T,S}^{\dag,(2)}=Q_AA_{T,S}^{(2)}$ & MPO &  $Q_AA_{T,S}^{(2)}A$ & $AA_{T,S}^{(2)}$ &  \cite{MoSt}  \\ \addlinespace
$A_{T,S}^{\dag,(2),\dag}=A_{T,S}^{\dag,(2)}P_A$ & MPOMP &  $Q_A A_{T,S}^{(2)}A$ & $AA_{T,S}^{(2)}P_A$ &  \cite{MoSt} \\ \addlinespace \hline 
\end{tabular}}
\caption{Generalized bilateral inverses}
\label{Generalized inverses 1}
\end{table}

\begin{table}[ht!]
\centering
\resizebox{12cm}{!} {
\begin{tabular}{l l l l l}
\hline
\textbf{Non generalized bilateral inverses} &  \textbf{Name} & \textbf{$XA=E$} &   \textbf{$AX=F$}  & \textbf{Reference} \\  \hline \addlinespace
$A^{\diamond }=(AP_{A})^{\dag }$  & BT  & $(AP_{A})^{\dag }A$  & $P_{A^2}$ & \cite{BaTr2, FeMa-chapter}\\
\addlinespace
$A^{\odagger}=A^{d}P_{A^{k}}$ & CEP & $A^{d}P_{A^{k}}A$ & $P_{A^{k}}$ &  \cite{MoPr, FeLeTh2}   \\
\addlinespace
 $A_{\odagger}=Q_{A^{k}}A^{d}$ & $*$CEP & $Q_{A^k}$ & $AQ_{A^{k}}A^{d}$ &  \cite{MoSt} \\
 \addlinespace
$A^{\weak}=(A^{\odagger})^2 A$ &  WG  &  $(A^{\odagger})^2 A^2$ & $A(A^{\odagger})^2 A$ &  \cite{WaCh}  \\ 
\addlinespace
$A^{\weak_2}=(A^{\odagger})^3 A^2$ &  GG  &  $(A^{\odagger})^3 A^3$ & $A(A^{\odagger})^3 A^2$ &  \cite{FeMa4}  \\ 
\addlinespace
 $A^{\weak_m}=(A^{\odagger})^{m+1} A^m$ &  $m$-WG  &  $(A^{\odagger})^{m+1} A^{m+1}$ & $A(A^{\odagger})^{m+1} A^m$ &  \cite{ZhChZh}  \\ 
 \addlinespace
 $A^{\core_m}= A^{\weak_m}P_{A^m}$ &  $m$-WC  &  $A^{\weak_m}P_{A^m}A$ & $AA^{\weak_m}P_{A^m}$ &  \cite{FeMa5}  \\ 
 \addlinespace
$A_{T,S}^{k,(2),\dagger}=A_{T,S}^{(2)}P_{A^k}$ & $k$-OMP &  $A_{T,S}^{(2)}P_{A^k}A$ & $AA_{T,S}^{(2)}P_{A^k}$ &  \cite{MoSt} \\ 
\addlinespace
$A_{T,S}^{k,\dagger,(2)}=Q_{A^k}A_{T,S}^{(2)}$ & $k$-MPO &  $Q_{A^k}A_{T,S}^{(2)}A$ & $AQ_{A^k}A_{T,S}^{(2)}$ &  \cite{MoSt} \\ \addlinespace \hline 
\end{tabular}}
\caption{Non generalized bilateral inverses}
\label{Generalized inverses 2}
\end{table}

\section*{Author Contribution} 

All authors contributed equally to the writing of this paper. All authors read and approved the final manuscript.

\section*{Data Availability} 

No data was used.

\section*{Funding} 

D.E. Ferreyra, F.E. Levis, and R.P. Moas are partially supported by Universidad Nacional de Río Cuarto (PPI 18/C559), Universidad Nacional de La Pampa, Facultad de Ingenier\'ia (Resol. Nro. 135/19) and CONICET (PIBAA 28720210100658CO). H.H. Zhu is supported by the National Natural Science Foundation of China (No. 11801124).

\section*{Declarations} 
 
\noindent {\bf Conflict of interest} The authors have no conflicts of interest.


\begin{thebibliography}{9}

\bibitem{BaTr} Baksalary, O.M., Trenkler, G.: Core inverse of matrices. Linear Multilinear Algebra \textbf{58} (6),  681-697 (2010)

\bibitem{BaTr2} Baksalary, O.M., Trenkler, G.: On a generalized core inverse. Appl. Math. Comput.  \textbf{236}, 450-457 (2014)

\bibitem{BeGr} Ben-Israel, A., Greville,  T.N.E.: Generalized Inverses: Theory and Applications. Second ed., Springer-Verlag, New York (2003)

\bibitem{ChMoXu}  Chen, J.L., Mosi\'c, D., Xu,  S.Z.: On a new generalized inverse for Hilbert space operators. Quaest. Math. \textbf{43} (9), 1331-1348 (2020)

\bibitem{Dr2}  Drazin, M.P.: A class of outer generalized inverses. Linear Algebra Appl. \textbf{436} (7), 1909-1923 (2012)

\bibitem{FeLeTh1} Ferreyra, D.E., Levis, F.E., Thome, N.: Revisiting of the core EP inverse and its extension to rectangular matrices. Quaest. Math.  \textbf{41} (2), 265-281 (2018)

\bibitem{FeLeTh2} Ferreyra, D.E., Levis, F.E., Thome, N.: Maximal classes of matrices determining generalized inverses. Appl. Math. Comput.  \textbf{333}, 42-52 (2018)

\bibitem{FeLePrTh}  Ferreyra, D.E., Levis, F.E., Priori, A.N., Thome, N.: The weak core inverse. Aequat. Math. \textbf{95}, 351–373 (2021)

\bibitem{FeMa-chapter}  Ferreyra, D.E., Malik, S.B.: The BT inverse. In: Kyrchei, I. (ed.) Generalized Inverses: Algorithms and Applications, pp 49-76. Nova Science Publishers, New York (2022)

\bibitem{FeMa4} Ferreyra, D.E., Malik, S.B.: A generalization of the group inverse.  Quaest. Math. \textbf{46} (10), 2129-2145 (2023)

\bibitem{FeMa5}  Ferreyra, D.E., Malik, S.B.: The $m$-weak core inverse. Rev. R. Acad. Cienc. Exactas Fis. Nat. Ser. A Mat.  \textbf{118}, Paper No. 41 (2024)

\bibitem{KhSa} Kheirandish, E., Salemi, A.: Generalized bilateral inverses.  J. Comput. Appl. Math. \textbf{428}, Paper No. 115137 (2023)

\bibitem{Ku} Kurata, H.: Some theorems on the core inverse of matrices and the core partial ordering. App. Math. Comput.  \textbf{316},  43-51 (2018)

\bibitem{MoPr}  Manjunatha Prasad, K., Mohana, K.S.: Core EP inverse. Linear Multilinear Algebra  \textbf{62} (6), 792-802 (2014)

\bibitem{MaTh} Malik, S.B., Thome, N.: On a new generalized inverse for matrices of an arbitrary index. Appl. Math. Comput.  \textbf{226}, 575-580 (2014)

\bibitem{Ma} Malik, S.B.: Some more properties of core partial order. App. Math. Comput. \textbf{221}, 192-201 (2013)

\bibitem{Mary}  Mary, X.: On generalized inverses and Green's relations. Linear Algebra Appl. \textbf{434} (8), 1836-1844 (2011)

\bibitem{MeSa} Mehdipour, M., Salemi, A.: On a new generalized inverse of matrices. Linear Multilinear Algebra \textbf{66} (5), 1046-1053 (2018)

\bibitem{MoSt} Mosi\'c, D., Stanimirovi\'c, P.S.: Composite outer inverses for rectangular matrices. Quaest. Math. \textbf{44} (1), 45-72 (2021)

\bibitem{Pe} Penrose, R.: Generalized inverse for matrices. Math. Proc. Cambridge Philos. Soc. \textbf{51} (3), 406-413 (1955)

\bibitem{Ra}  Raki\'c, D.S.: A note on Rao and Mitra's constrained and Drazin's $(b,c)$ inverse. Linear Algebra Appl. \textbf{523}, 102-108 (2017)

\bibitem{RaMi} Rao, C.R., Mitra, S.K.: Generalized inverse of a matrix and its application. In: Proceedings of the Sixth Berkeley Symposium on Mathematics, Statistics and Probability, vol. 1, pp. 601–620. University of California Press, Berkeley (1972)

\bibitem{WaCh} Wang, H., Chen, J.: Weak group inverse. Open Math. \textbf{16} (1),  1218-1232 (2018)

\bibitem{WaLi} Wang H., Liu, X.: Characterizations of the core inverse and the core partial ordering. Linear Multilinear Algebra \textbf{63} (9), 1829-1836 (2015)

\bibitem{ZhChZh}  Zhou Y., Chen, J., Zhou, M.: $m$-weak group inverses in a ring with involution. Rev. R. Acad. Cienc. Exactas Fis. Nat. Ser. A Mat. \textbf{115}, Paper No. 2 (2021)

\end{thebibliography}
\end{document}